\documentclass[a4paper,12pt,reqno]{amsart}
\usepackage[T1]{fontenc}
\usepackage{amssymb}
\usepackage{a4}

\newcommand{\abs}[1]{\left|#1\right|}

\newcommand{\set}[1]{\left\{#1\right\}}
\newcommand{\bra}[1]{\left(#1\right)}
\newcommand{\Com}{\mathbb C}

\newcommand{\eps}{\varepsilon}
\newcommand{\B}{\mathcal{B}_\rho}
\newcommand{\Q}{\mathcal{Q}}
\newcommand{\A}{\mathcal{A}}

\newtheorem{theo}{Theorem}[section]
\newtheorem{lemm}[theo]{Lemma}

\newtheorem{defi}[theo]{Definition}

\begin{document}

\author{Jörn Peter}
\title[Hausdorff measure]{Hausdorff measure of escaping and Julia sets for bounded type functions of finite order}
\thanks{Supported by the Deutsche Forschungsgemeinschaft, Be 1508/7-1.}

\begin{abstract}
We show that the escaping sets and the Julia sets of bounded type transcendental entire functions of order $\rho$ become 'smaller' as $\rho\to\infty$. More precisely, their Hausdorff measures are infinite with respect to the gauge function $h_\gamma(t)=t^2g(1/t)^\gamma$, where $g$ is the inverse of a linearizer of some exponential map and $\gamma\geq(\log\rho(f)+K_1)/c$, but for $\rho$ large enough, there exists a function $f_\rho$ of bounded type with order $\rho$ such that the Hausdorff measures of the escaping set and the Julia set of $f_\rho$ with respect to $h_{\gamma'}$ are zero whenever $\gamma'\leq(\log\rho-K_2)/c$.
\end{abstract}

\maketitle

\section{Main results and outline}
\subsection{Introduction and main result}
Let $f$ be a transcendental entire function.
The \emph{Julia set} $\mathcal{J}(f)$ is the set of points in $\Com$ where the iterates
$f^n$ do not form a normal family with respect to the spherical metric on $\Com\cup\{\infty\}$, the \emph{escaping set} $I(f)$ is the set of all points $z$ such that $f^n(z)$ tends to infinity as $n\to\infty$.
Eremenko \cite{eremenko} showed that $\mathcal{J}(f)=\partial I(f)$.
Let $S(f)$ denote the set of \emph{singular values} of $f$, which is by definition
the smallest closed set $A$ such that $f:\Com\setminus f^{-1}(A)\to \Com\setminus A$ is a covering map. It can be easily
verified that the set of singular values is the closure of the set of critical and finite asymptotic values of $f$.
We say that $f$ is of \emph{bounded type} if $S(f)$ is bounded, and denote the set of all functions of bounded type by $\mathcal{B}$. In \cite{eremenkolyubich}, Eremenko and Lyubich introduced the \emph{logarithmic change of coordinates} which has become a standard tool to investigate properties of bounded type entire functions. Using this technique, they showed that $I(f)\subset\mathcal{J}(f)$, and hence $\overline{I(f)}=\mathcal{J}(f)$ by Eremenko's result, for every $f\in\mathcal{B}$.
The \emph{order} of an entire function $f$ is defined as
\begin{equation}\label{eqeq}
\rho(f):=\limsup_{r\to\infty}\frac{\log\log M(r,f)}{\log r}.
\end{equation}
Here, $M(r,f):=\max_{\abs{z}=r}\abs{f(z)}$.
If $\rho(f)<\infty$, then we say that $f$ is of \emph{finite order} $\rho(f)$.
From now on, we use the notation $\B:=\{f\in\mathcal{B}:f\text{ has finite order }\rho\}$. Note that if $f\in\mathcal{B}$ then $\rho(f)\geq 1/2$ (see for example \cite{aspenbergbergweiler} for an argument).

We examine the Hausdorff measure of escaping and Julia sets of functions $f\in\mathcal{B}$ of finite order with respect to certain gauge functions.
By a \emph{gauge function}, we mean an increasing function $h:[0,\eps)\to\mathbb{R}_{\geq 0}$ (where $\eps>0$) which is continuous from the right and satisfies $h(0)=0$. For an arbitrary set $A\subset\Com$, define
\begin{equation*}\label{b1}
\mathcal{H}^h(A):=\lim_{\delta\to 0}\inf\set{\sum_{i=1}^\infty h(\text{diam
}A_i):\bigcup_{i=1}^\infty A_i\supset A,\text{ diam
}A_i<\delta\text{ for every }i}
\end{equation*}
Then $\mathcal{H}^h$ is a metric outer measure on all subsets of $\Com$, called
the \emph{Hausdorff measure with respect to $h$}. Following \cite{rogers}, we introduce the notation $h_1\prec h_2$ for gauge functions $h_1$ and $h_2$ whenever the quotient $h_1(t)/h_2(t)$ tends to 0 as $t\to 0$.
In the special case where $h(t)=h^s(t):=t^s$ for some $s>0$, $\mathcal{H}^{h^s}$ is the
\emph{$s$-dimensional outer Hausdorff measure}. Given $A\subset\Com$, it is well known that there exists $s_0\geq0$ such that $\mathcal{H}^{h^s}(A)=\infty$
if $s<s_0$ and $\mathcal{H}^{h^s}(A)=0$ if $s>s_0$. This value $s_0$ is called the \emph{Hausdorff dimension} of the set $A$, which we will denote by HD$(A)$.
Bara\'{n}ski \cite{baranski} and (independently) Schubert \cite{schubert} showed that HD$(\mathcal{J}(f))=2$ whenever $f\in\B$. In fact, the stronger result HD$(I(f))=2$ also holds. However, if the order of $f$ is infinite, this need not be true anymore, as was shown by Stallard \cite{stallard}. In \cite{bks}, Bergweiler, Karpi\'{n}ska and Stallard proved that if the order of $f$ is infinite and $M(r,f)\leq\exp(\exp((\log r)^{q+\eps}))$ for large $r$, then HD$(\mathcal{J}(f))\geq 1+\frac{1}{q}$, and this estimate is sharp \cite{stallard}. This suggests that the escaping set and the Julia set of a function $f\in\B$ get 'smaller' as $\rho$ increases.
On the other hand, a result by Eremenko and Lyubich \cite[Proposition 4 and Theorem 7]{eremenkolyubich} implies that if $f$ has finite order and a logarithmic singularity, then $I(f)$ has zero two-dimensional Hausdorff measure; there are many functions satisfying this condition, so the usual $s$-dimensional Hausdorff measure is not suitable to distinguish sizes of escaping sets of bounded type entire functions with finite order, which is why we use more general gauge functions to measure them.
This question was addressed for the exponential functions $E_\lambda(z):=\lambda\exp(z)$ in \cite{peter}.
Let $\lambda\in(0,1/e)$ and $E_\lambda(z):=\lambda\exp(z)$ be the exponential map with parameter $\lambda$.
The function $E_\lambda$ has exactly one real repelling fixed point $\beta_\lambda$, that is, $E_\lambda(\beta_\lambda)=\beta_\lambda$ and $E_\lambda'(\beta_\lambda)>1$. A classical result
 due to K\oe nigs and Poincaré implies that there exists an entire function $L_\lambda$ satisfying $L_\lambda(0)=\beta_\lambda$, $L_\lambda'(0)=1$ and
\begin{equation*}
E_\lambda(L_\lambda(z))=L_\lambda(\beta_\lambda z)\text{ for all }z.
\end{equation*}
The proof of this theorem (and of some other results stated here without proof) can for example be found in
\cite{beardon}, \cite{bergweiler}, \cite{milnor} or \cite{steinmetz}.
Set $\Phi_\lambda:=(L_\lambda|_\mathbb{R})^{-1}$. It is easy to see that $\Phi_\lambda(x)$ tends
to $\infty$ as $x\to\infty$, but slower than any iterate of the logarithm.
Set $h_{\lambda,\gamma}(t):=t^2\Phi_\lambda(1/t)^\gamma$ for $\gamma>0$. The function $h_{\lambda,\gamma}$ is defined on a small interval $(0,\eps)$ and can be continued continuously to 0 by $h(0):=0$. It can be easily verified that $h_{\lambda,\gamma}$ is a gauge function. It was shown
in \cite{peter} that there exists $K_lambda$ such that
$\mathcal{H}^{h_{\lambda,\gamma}}(\mathcal{J}(E_\mu))=\mathcal{H}^{h_{\lambda,\gamma}}(I(E_\mu))=\infty$ for $\gamma>K_lambda$ and all $\mu$ while
$\mathcal{H}^{h_{\lambda,\gamma}}(\mathcal{J}(E_\mu))=\mathcal{H}^{h_{\lambda,\gamma}}(I(E_\mu))=0$ for $\gamma<K_lambda$ for $\mu$ such that $E_\mu$ has an attracting periodic point. Here
we obtain estimates of this type for all functions of finite
order in $\mathcal{B}$, with the exponent $\gamma$ depending on the order.

\begin{theo}\label{main1}
There exists $K_1>0$ with the following property: If $\rho\geq 1/2$ and $f\in\B$, then $\mathcal{H}^{h_{\lambda,\gamma}}(I(f))=\infty$ whenever $\gamma>(\log(\rho)+K_1)/\log\beta_\lambda$.
\end{theo}
Since $I(f)\subset\mathcal{J}(f)$ as mentioned in the introduction, this theorem immediately implies that $\mathcal{H}^{h_{\lambda,\gamma}}(\mathcal{J}(f))=\infty$ whenever $\gamma>(\log(\rho)+K_1)/\log\beta_\lambda$. The second result shows that for $\rho$ large enough, Theorem \ref{main1} is sharp (in the sense described below).
\begin{theo}\label{main2}
There exist $K_2,\rho_0>0$ with the following property: If $\rho\geq\rho_0$, there exists $f_\rho\in\B$ such that $\mathcal{H}^{h_{\lambda,\gamma}}(\mathcal{J}(f_\rho))=0$ for $\gamma<(\log(\rho)-K_2)/\log\beta_\lambda$.
\end{theo}
Again, it follows immediately that the statement is still valid if we substitute $\mathcal{J}(f)$ by $I(f)$.

Summarizing, these results can be interpreted as follows: If $f$ does not grow faster than $\exp(\abs{z}^{\rho+\eps})$ for every $\eps$, then the Hausdorff measure of $f$ is infinite with respect to the function $h_{\lambda,\gamma(\rho,\lambda)}$, and $\gamma(\rho,\lambda)$ necessarily has to increase with $\rho$ (if we keep $\lambda$ fixed). This means that, the higher the order of $f$ is, the 'smaller' are $I(f)$ and $\mathcal{J}(f)$.

We also mention here that similar, but sharper results were proved for the exponential family in \cite{peter}. Here, we use techniques very different from those in \cite{peter} - for the proof of Theorem \ref{main1}, this seems to be clear, since the class of functions under consideration is much more general here. But also the methods to prove Theorem \ref{main2} are very different from those that we applied to show the corresponding result for the exponential family. Although the function $f_\rho$ that we construct in the proof of Theorem \ref{main2} morally behaves like $\exp(z^\rho)$, it has zeros and critical points arbitrarily close to the boundary of the tract $W$ of $f_\rho$. This makes it impossible to find a fixed radius $r$ such that for every $z\in\log W$, the logarithmic transform $F_\rho$ of $f_\rho$ can be continued analytically to a disk of radius $r$ around $z$. Hence the distortion of $F_\rho$ near the boundary of a logarithmic tract is not small, unlike as in the exponential family. This fact gives the need to use different methods than in \cite{peter} also for the proof of Theorem \ref{main2}, although the main idea remains the same. (For the definitions of 'tract', 'logarithmic transform' and 'distortion', see section 2.2.1).

This paper is organized as follows:
In section 2, we provide some notations that we will use throughout this work and we briefly review some
classical results from function theory as well as holomorphic dynamics, like Koebe's
distortion theorem  or the logarithmic transform of a function in class $\mathcal{B}$.
Finally, we mention some results for the functions $h_{\lambda,\gamma}$ that will be used in the proofs of the two main theorems.
Sections 3 and 4 contain the proofs of Theorems \ref{main1} and \ref{main2}, respectively,
together with the necessary preparations.
I thank Walter Bergweiler and Helena Mihaljevi\'{c}-Brandt for many fruitful discussions about this project.

\section{Notations and preliminaries}

\subsection{Notations}

For $z \in \Com$, let $\Re z$ and $\Im
z$ denote the real and imaginary parts of $z$, respectively. If $z_0 \in \Com$
and $r>0$, we write $D(z_0,r)$ for the disk in $\Com$ with center
$z_0$ and radius $r$ with respect to the euclidean metric. By $\mathbb{D}:=D(0,1)$ we denote the open unit disk in $\Com$, and by $\mathbb{H}:=\{z\in\Com:\Re z>0\}$ the right half plane.
For $r>0$ and $\theta\in\mathbb{R}$, let \[Q(0,r,0):=\set{z\in\Com:\max\set{\abs{\Re z},\abs{\Im z}}<\frac{r}{2}}\]
and \[Q(z_0,r,\theta)=z_0+e^{i\theta}Q(0,r,0).\]
If the angle $\theta$ is not important, we will suppress it and just write
$Q(z_0,r)$ in order to increase readability.

We denote the two-dimensional Lebesgue measure of a measurable set $A\subset\Com$ by $\abs{A}$ and the one-dimensional Lebesgue measure of a measurable set $A\subset\mathbb{R}$ by $\mathcal{L}(A)$.
If $A,B \subset \Com$ and $0<\abs{B}<\infty$, we write dens($A,B$) for the
density of $A$ in $B$, which is defined by
\[\text{dens}(A,B):=\frac{\abs{A \cap B}}{\abs{B}}.\]
Let $f$ be an entire function and $S(f)$ be the set of singular values of $f$, that is, $S(f)$ is the closure of
the set of critical and finite asymptotic values of $f$.
We denote the \emph{postsingular set} of $f$ by $P(f)$, which is by definition
\[P(f):=\overline{\bigcup_{n\in\mathbb{N}_0}f^n(S(f))}.\]
Recall from the introduction that $\mathcal{B}:=\{f\text{ transcendental, entire}:S(f)\text{ bounded}\}$ and $\B:=\{f\in\mathcal{B}: f\text{ has finite order }\rho\},$ where the order $\rho(f)$ is defined by \eqref{eqeq}.
In the proofs, $\eps_1,\eps_2,\ldots$ refer to positive real numbers which can be arbitrarily small if other quantities have been chosen suitably. As an example, we write 'Let $R$ be large. Then $M(r,f)\leq\exp(r^{\rho(f)+\eps_1})$ for $r>R$' without emphasizing that $R$ in fact depends on $\eps_1$.

\subsection{Preliminaries}
The results of this section will mainly be stated without proofs; we will give references where needed.

\subsubsection{The logarithmic transform}

This is a standard tool in complex dynamics ever since it was introduced by Eremenko and Lyubich
\cite{eremenkolyubich}. Let $f\in\mathcal{B}$ and assume that
$S(f)\subset\mathbb{D}$ and $f(0)\in\mathbb{D}$ (this can always be achieved by
conjugating $f$ with a conformal automorphism of $\Com$). Then the number of components of $f^{-1}(\Com\setminus\nolinebreak\overline{\mathbb{D}})$, also called \emph{tracts} of $f$, is finite by the Denjoy-Carleman-Ahlfors theorem (see for example \cite{nevanlinna}). Let $W_1(f),\ldots,W_n(f)$ be the tracts of $f$. Eremenko and Lyubich \cite{eremenkolyubich} showed that every $W_i(f)$ is simply connected and bounded by an analytic curve that tends to $\infty$ on both ends, and $f\circ\exp:T_i^k(f)\to\Com\setminus\overline{\mathbb{D}}$ is a universal covering for every component $T_i^k(f)$ of $\log W_i(f)$. Clearly, the $T_i^k(f)$ can be related to each other by fixing one component $T_i^0(f)$ and setting $$T_i^k(f):=T_i^0(f)+2\pi ik.$$ Such a component $T_i^k(f)$ is called \emph{logarithmic tract} of $f$. The map $\exp:\mathbb{H}\to\Com\setminus\overline{\mathbb{D}}$ is also a universal covering. So, with $$\mathcal{T}(f):=\bigcup_{i,k}T_i^k(f),$$ there exists a map $F:\mathcal{T}(f)\to\mathbb{H}$ such that $$F_{T_i^k(f)}:=F|_{T_i^k(f)}$$ is a conformal isomorphism for every $i$ and every $k$.
We say that $F$ is the \emph{logarithmic transform} of $f$ and abbreviate
\begin{equation}\label{G}
G_{T_i^k(f)}:=\left(F_{T_i^k(f)}\right)^{-1}.
\end{equation}
The function $F$ has an expanding property that follows from the Koebe one-quarter theorem and the fact that every logarithmic tract does not contain disks of radius bigger than $\pi$:
\begin{lemm}[\cite{eremenkolyubich}, Lemma 1]
For every $z\in\mathbb{H}$ and every $T\in\mathcal{T}(f)$, we have
\begin{equation}\label{a1}
\abs{(F_T^{-1})'(z)}\leq\frac{4\pi}{\Re z}.
\end{equation}
\end{lemm}
In this paper, we will rather prove the results for $F$ than for $f$, so we have to define the Julia and escaping sets of $F$. Denote by $$\mathcal{J}(F):=
\{z\in\Com:z\in\mathcal{T}(f),\exp(z)\in\mathcal{J}(f)\}$$ the \emph{Julia set of $F$} and
by $$I(F):=\{z\in\Com:\Re F^n(z)\to\infty\text{ as }n\to\infty\}$$ the \emph{escaping set of $F$}.
Obviously, $\exp(\mathcal{J}(F))\subset\mathcal{J}(f)$ and $\exp(I(F))\subset I(f)$.
Further, $\exp(\mathcal{J}(F))=\mathcal{J}(f)$ if $\mathbb{D}\cap\mathcal{J}(f)=\emptyset$. This will be important in the proof of Theorem \ref{main2}.

\subsubsection{Koebe's theorem and distortion}
We begin this section by stating the classical Koebe theorem.
\begin{theo} \label{kk}
Let $z_0 \in \Com$, $r>0$, $f:D(z_0,r) \to \Com$ be a univalent
function and $z \in D(z_0,r)$. Then
\begin{equation*}
\frac{r^2\abs{f'(z_0)}(r-\abs{z-z_0})}{(r+\abs{z-z_0})^3}\leq\abs{f'(z)}\leq
\frac{r^2\abs{f'(z_0)}(r+\abs{z-z_0})}{(r-\abs{z-z_0})^3}
\end{equation*}
and
\begin{equation*}
\frac{r^2\abs{f'(z_0)}\abs{z-z_0}}{(r+\abs{z-z_0})^2}\leq\abs{f(z)-f(z_0)}\leq
\frac{r^2\abs{f'(z_0)}\abs{z-z_0}}{(r-\abs{z-z_0})^2}.
\end{equation*}
\end{theo}
Koebe's theorem implies in particular that the family of all univalent functions $f:\mathbb{D}\to\Com$
with $f(0)=0$ and $f'(0)=1$ is normal. This yields the following result.
\begin{theo} \label{epsdelta}
For every $\eps>0$, there exists $\delta>0$ such that if $f:\mathbb{D}\to\Com$ is univalent
with $f(0)=0$ and $f'(0)=1$, then
\[\abs{\frac{f(z)}{z}-1}<\eps\]
whenever $\abs{z}<\delta$.
\end{theo}
Let $U\subset\Com$ either be open and bounded or the closure of such a set.
A function $f:U\to\Com$ is said to have \emph{bounded distortion} if $f$ is a bilipschitz mapping, that is,
\[0<c_f:=\inf_{\substack{z,w \in U\\z\neq w}}\frac{\abs{f(z)-f(w)}}{\abs{z-w}}\leq
\sup_{\substack{z,w \in U\\z\neq
w}}\frac{\abs{f(z)-f(w)}}{\abs{z-w}}=:C_f<\infty.\]
The distortion of $f$ is then defined as $D(f):=C_f/c_f$.
It can be easily shown that if $U$ is open and $f$ has bounded distortion, then $f$ extends to a function on $\overline{U}$ with the same distortion as $f$. Conversely, if $U$ is the closure of an open bounded set, then $f|_{\text{int}(U)}$ has the same distortion as $f$.
Note that the distortion of a holomorphic function $f:U\to\Com$ is often defined by $$L(f):=\frac{\sup_{z\in U}\abs{f'(z)}}{\inf_{z\in U}\abs{f'(z)}}.$$ It is very easy to see that $L(f)\leq D(f)$ for every function $f$, but in
general, we do not have equality. If both $U$ and $f(U)$ are convex, it can be shown that $D(f)=L(f)$. But even if $U$ is convex and $f$ is univalent on $U$, $L(f)$ may be finite but $f$ does not have bounded distortion. An example is given by $f(z)=z^4$, defined on all $z=x+iy$ such that $4(x-1)^2+y^2<1$. In fact, $1+i,1-i\in\partial U$ and $(1+i)^4=(1-i)^4$, so $f$ does not have bounded distortion. On the other hand, $\sup_{z\in U}\abs{z}=\sqrt{2}$ and $\inf_{z\in U}\abs{z}=1/2$, so $L(f)<16\sqrt{2}$. However, using the Koebe theorems, it is easy to see that $D(f)$ is bounded whenever $f:U\to\Com$ can be continued univalently to a domain which compactly contains $U$.
\begin{lemm} \label{first}
Let $z_0 \in \Com,r>0,K>1$. Let $f:D(z_0,Kr)\to\Com$ be a
univalent function. Then
\[L(f|_{D(z_0,r)})\leq\left(\frac{K+1}{K-1}\right)^4\] and
\[D\left(f|_{D(z_0,r)}\right)\leq L(f)\left(\frac{K+1}{K-1}\right)^2\leq\left(\frac{K+1}{K-1}\right)^6.\]
\end{lemm}
\begin{proof}
Without loss of generality, we may assume that $z_0=0$ and $r=1$.
By Theorem \ref{kk}, we have
$$\abs{f'(z)}\leq K^2\abs{f'(z_0)}\frac{K+\abs{z-z_0}}{(K-\abs{z-z_0})^3}\leq K^2\abs{f'(z_0)}\frac{(K+1)}{(K-1)^3}$$
and analogously $$\abs{f'(z)}\geq\abs{f'(z_0)}\frac{K^2(K-1)}{(K+1)^3}$$ for all $z\in\mathbb{D}$. It follows that $L(f)\leq\left(\frac{K+1}{K-1}\right)^4$.
To prove the second statement, note that $$\frac{\abs{f(z)-f(w)}}{\abs{z-w}}\leq\sup_{y\in \mathbb{D}}\abs{f'(y)}$$ for all $z,w\in \mathbb{D}$. To prove an inequality in the other direction, let $z,w\in \mathbb{D}$, and define
$$\varphi:\mathbb{D}\to D(0,K),u\mapsto\frac{z-Ku}{1-\frac{u\overline{z}}{K}}.$$
Then $\varphi$ is biholomorphic and $\varphi(0)=z$. Let $v:=\varphi^{-1}(w)$. We have
\begin{align*}
\abs{f(w)-f(z)}&=\abs{(f\circ\varphi)(v)-(f\circ\varphi)(0)}\geq\abs{(f\circ\varphi)'(0)}\frac{\abs{v}}{(1-\abs{v})^2}\\
&=\abs{f'(z)}\abs{\varphi'(0)}\frac{\abs{v}}{(1-\abs{v})^2}\geq\inf_{y\in \mathbb{D}}\abs{f'(y)}\frac{1}{K}(K^2-\abs{z}^2)\frac{\abs{v}}{(1-\abs{v})^2}\\
&\geq\inf_{y\in \mathbb{D}}\abs{f'(y)}\frac{K^2-1}{K}\frac{\abs{v}}{(1-\abs{v})^2}\\
&\geq\inf_{y\in \mathbb{D}}\abs{f'(y)}\frac{K^2-1}{K}\frac{K\abs{w-z}}{\abs{\overline{z}w-K^2}}\frac{1}{(1+\frac{K\abs{w-z}}{\abs{\overline{z}w-K^2}})^2}\\
&\geq\inf_{y\in \mathbb{D}}\abs{f'(y)}(K^2-1)\abs{w-z}\frac{\abs{\overline{z}w-K^2}}{(\abs{\overline{z}w-K^2}+K\abs{w-z})^2}\\
&\geq\inf_{y\in \mathbb{D}}\abs{f'(y)}(K^2-1)\abs{w-z}\frac{K^2-1}{(K^2+1+2K)^2}\\
&\geq\inf_{y\in \mathbb{D}}\abs{f'(y)}\abs{w-z}\frac{(K-1)^2}{(K+1)^2}
\end{align*}
by Theorem \ref{kk}, so $D(f)\leq L(f)(\frac{K+1}{K-1})^2$.
\end{proof}
Some simple properties of holomorphic functions with bounded distortion are summarized in the following lemma.
\begin{lemm} \label{eigensch}
Let $U \subset \Com$ be bounded and open. Let $f:U \to \Com$ be a univalent
function with bounded distortion. Then the following statements
hold:
\begin{enumerate}
\item $D(f)=D(f^{-1})$, $L(f)=L(f^{-1})$ \\
\item If $V\supset f(U)$ is a domain and $g:V\to\Com$ is
univalent with bounded distortion, then $D(g\circ f)\leq
D(g)D(f)$ and $L(g\circ f)\leq L(g)L(f)$. \\
\item If $A\subset U$ is Lebesgue-measurable, then $f(A)$ is
Lebesgue-measurable and dens$(A,U)\leq L(f)^2$dens$(f(A),f(U))\leq D(f)^2dens(f(A),f(U)).$
\end{enumerate}
\end{lemm}
If we apply a function $f$ with small distortion to a square, then its image is a set which is almost square-shaped. More precisely, we have the following lemma, which is an application of Theorem \ref{epsdelta}.
\begin{lemm}\label{second}
For every $1/\sqrt{2}>\eps>0$, there exists a constant $K>1$ which satisfies
the following property: Let $z_0 \in
\Com$, $r>0$, $K'\geq K$, a univalent function
$\tilde{f}:D\left(z_0,K'r/\sqrt{2}\right) \to \Com$ and a
square $Q=Q(z_0,r,\theta)$ be given. Let
$f:=\tilde{f}|_{\overline{Q}}$ and $d:=D(f)$ be the distortion of
$f$. Then
\[Q\left(f(z_0),\abs{f'(z_0)}r\frac{1}{d}\bra{1-\sqrt{2}\eps},\theta+\arg
f'(z_0)\right)\subset f(Q)\] and \[f(Q)\subset
Q\left(f(z_0),\abs{f'(z_0)}rd\bra{1+\sqrt{2}\eps},\theta+\arg
f'(z_0)\right).\]
\end{lemm}

\subsection{Two results for $h_{\lambda,\gamma}$}

Recall from the introduction that for $\lambda\in(0,1/e)$, the function $E_\lambda(z):=\lambda\exp(z)$ has a unique real repelling fixed point $\beta_\lambda$, and there is a function $\Phi_\lambda$ which satisfies
\begin{equation}\label{99}
\Phi_\lambda(E_\lambda(x))=\beta_\lambda \Phi_\lambda(x)
\end{equation}
for all $x\geq\beta_\lambda$. We consider the gauge function $$h_{\lambda,\gamma}(t):=t^2\Phi_\lambda(1/t)^\gamma.$$
The first result about $h_{\lambda,\gamma}$ that we mention is that the measure $\mathcal{H}^{h_{\lambda,\gamma}}$ essentially only depends on $\beta_\lambda^\gamma$ (see \cite{peter}).
\begin{theo}\label{unabh}
Let $\lambda_1,\lambda_2\in(0,1/e)$. If $\gamma_1,\gamma_2$ are chosen such that $\beta_{\lambda_1}^{\gamma_1}=\beta_{\lambda_2}^{\gamma_2}$, then there exist constants $c,C>0$ with $$ch_{\lambda_2,\gamma_2}(t)\leq h_{\lambda_1,\gamma_1}(t)\leq Ch_{\lambda_2,\gamma_2}(t)$$ if $t$ is small enough.
\end{theo}
Next, we show that zero and infinite $\mathcal{H}^{h_{\lambda,\gamma}}$-measure are preserved under bilipschitz mappings.
\begin{lemm}\label{preserved}
Let $A\subset\Com$ and $f$ be a bilipschitz mapping. If $\mathcal{H}^{h_{\lambda,\gamma}}(A)=0$ $($resp. $\infty)$, then $\mathcal{H}^{h_{\lambda,\gamma}}(f(A))=0$ $($resp. $\infty)$.
\end{lemm}
\begin{proof}
First note that for every $K>0$, there exists $K'(K)>0$ such that $h_{\lambda,\gamma}(Kt)\leq K'(K)h_{\lambda,\gamma}(t)$.
In fact, for $K\leq 1$, we can choose $K'(K)=K$. If $K>1$, we have with $K'(K)=K^2$ that
$$h_{\lambda,\gamma}(Kt)=K^2t^2\Phi_\lambda(1/(Kt))^\gamma\leq K^2t^2\Phi_\lambda(1/t)^\gamma=K'(K)h_{\lambda,\gamma}(t).$$
Suppose that $c<\abs{f(x)-f(y)}/\abs{x-y}<C$ for all $x,y$.
Let $\mathcal{H}^{h_{\lambda,\gamma}}(A)=0$ and $\{A_i\}$ be a covering for $A$. Then $\{f(A_i)\}$ is a covering for $f(A)$ and $$\sum h_{\lambda,\gamma}(\text{diam }f(A_i))\leq\sum h_{\lambda,\gamma}(C\cdot\text{diam }A_i)\leq K'(C)\sum h_{\lambda,\gamma}(\text{diam }A_i).$$
If $\mathcal{H}^{h_{\lambda,\gamma}}(A)=\infty$, let $\{B_i\}$ be a covering for $f(A)$. Then $\{f^{-1}(B_i)\}$ is a covering for $A$ and we have
$$\sum h_{\lambda,\gamma}(\text{diam }B_i)\geq\sum h_{\lambda,\gamma}(c\cdot\text{diam }f^{-1}(B_i))\geq\frac{1}{K(1/c)}\sum h_{\lambda,\gamma}(\text{diam }f^{-1}(B_i)).$$ Since for diam($f(U)$) is bounded above and below by a multiple of diam$(U)$ which is independent of $U\subset A$, the lemma is proved.
\end{proof}

\section{The estimate from below}

\subsection{Preparations}
Since the fundamental work of McMullen \cite{mcmullen}, there is a standard method
for estimating the Hausdorff measure of Julia sets of transcendental entire functions.
We will only give a very brief introduction here.
\begin{defi}[nesting conditions]\label{nest}
For $n \in \mathbb{N}$, let $\A_n$ be a finite collection
of compact, disjoint and connected subsets of $\Com$ with positive
Lebesgue-measure. Let $A_n$ be the union of the elements of $\A_n$. The intersection $$A:=\bigcap_{n=0}^\infty A_n$$ is a non-empty and compact set. We say that the sequence $(\A_n)$ satisfies the \emph{nesting conditions} if it has the following three properties:
\begin{enumerate}
\item For every $n \in \mathbb{N}$ and every $B \in \A_n$, there exists some $B' \in
\A_{n-1}$ such that $B \subset B'$.
\item There exists a decreasing sequence $(d_n)$
converging to 0 with \[\max_{B\in\A_n}\text{diam}(B)\leq d_n\] for all $n\in\mathbb{N}$.
\item There exists a sequence $(\Delta_n)$ of positive real numbers with \[\text{dens}(A_{n+1},B)\geq\Delta_n\] for all $n\geq 0,B\in\A_n$.
\end{enumerate}
\end{defi}
The key lemma to the proof of Theorem \ref{main1} is
\begin{lemm} \label{infinity}
Let $\set{\A_n}$ be a collection of families of sets
which satisfies the nesting conditions $($with properly chosen
sequences $(d_n)$ and $(\Delta_n))$. Let $A$ be defined as above.
Let $\eps>0$ and $g:(0,\eps)\to\mathbb{R}_{\geq 0}$ be a
decreasing continuous function such that $t^2g(t)$ is increasing.
Further, suppose that
$\lim_{t\to 0}t^2g(t)=0$ and
\begin{equation}\label{69}
\lim_{n\to\infty} g(d_n)\prod_{j=1}^n\Delta_j=\infty.
\end{equation}
Define
\[h:[0,\eps)\to\mathbb{R},t\mapsto
\begin{cases}
t^2g(t)&,t>0\\
0&,t=0
\end{cases}\]
Then $\mathcal{H}^h(A)=\infty$.
\end{lemm}
The proof follows ideas of McMullen (\cite{mcmullen}, Proposition 2.2), it
can be found in \cite{peter}.

\subsection{Proof of Theorem \ref{main1}}

Let $F$ be the logarithmic transform of $f$. Since $\exp$ is bilipschitz on small disks, Lemma \ref{preserved} implies that it suffices to show that $\mathcal{H}^{h_{\lambda,\gamma}}(I(F))=\infty$.
Our goal is to construct $(\A_n)$ with $A=\bigcap A_n\subset I(F)$ and such that $(\A_n)$ satisfies the nesting conditions with sequences $(\Delta_n)$ and $(d_n)$ that meet the requirement \eqref{69} (where $g(t)=\Phi_\lambda(1/t)^\gamma$).
We will use the following result due to Aspenberg and Bergweiler \cite{aspenbergbergweiler}:
\begin{theo}\label{aspberg}
Let $g$ be entire and $W$ be a tract of $g$ such that $\{\abs{z}=r\}\not\subset W$ for all large $r$. Let $0<\beta<1/2$ and put
\begin{equation}\label{b2}
V_\beta:=\{z\in W:\abs{g(z)}\geq\exp(\abs{z}^\beta)\}
\end{equation}
and $\psi_{V_\beta}(r)=\mathcal{L}(\{t\in[0,2\pi]:re^{it}\in V_\beta\})$. Let $0<\kappa<1$. Then there exist constants $C,r_0>0$ such that
$$\log\log M(r,g)\geq\pi\int_{r_0}^{\kappa r}\frac{dt}{t\psi_{V_\beta}(t)}-C$$ for all $r\geq r_0/\kappa$.
\end{theo}
Fix a tract $W$ of $f$, $\tilde{\beta}\in(0,1/2)$ and $\tilde{\kappa}\in(0,1)$. Define $V_{\tilde{\beta}}$ as in \eqref{b2} and set $\tilde{\psi}:=\psi_{V_{\tilde{\beta}}}$.
By Theorem \ref{aspberg}, there exist constants $r_0>0$ and $C>0$ such that $$\pi\int_{e^{r_0}}^{\tilde{\kappa} e^r}\frac{dt}{t\tilde{\psi}(t)}-C\leq\log\log M(e^r,f)$$ whenever $\tilde{\kappa} e^r>e^{r_0}$. Let these constants be fixed. Let $T^0$ be a logarithmic tract of $f$ corresponding to $W$ and let $T^k:=T^0+2\pi ik$ for $k\in\mathbb{Z}$ (compare section 2).
Since $\log\log M(e^r,f)=\max_{\Re y=r}\log\Re F(y)$, we obtain with the transformation $\tilde{\theta}(r):=\tilde{\psi}(e^r)$ that
\begin{equation}\label{a}
\int_{r_0}^{r+\log\tilde{\kappa}}\frac{ds}{\tilde{\theta}(s)}\leq\frac{1}{\pi}\left(\max_{\Re y=r}\log\Re F(y)+C\right)
\end{equation}
whenever $r+\log\tilde{\kappa}\geq r_0$. It follows that
\begin{align*}
\int_{x_1}^{x_2}\frac{ds}{\tilde{\theta}(s)}&\leq\int_{r_0}^{x_2}\frac{ds}{\tilde{\theta}(s)}\leq\frac{1}{\pi}\left(\max_{\Re y=x_2-\log\tilde{\kappa}}\log\Re F(y)+C\right)\\
&\leq\frac{1}{\pi}\left((\rho(f)+\eps_1)(x_2-\log\tilde{\kappa})+C\right)\leq\frac{1}{\pi}\left((\rho(f)+\eps_1)x_2+C'\right)
\end{align*}
for $x_1\leq x_2-\log\tilde{\kappa}$, where $C'$ may depend on $\tilde{\kappa},\tilde{\beta},f$ and $W$.
Now we choose $\lambda\in(0,1/e)$ so small that $R_0:=3\beta_\lambda/\lambda$ is much larger than $C', r_0$ and $-\log\tilde{\kappa}$, and define inductively $$R_n:=\min\{2^kR_{n-1}:2^kR_{n-1}\geq e^{\lambda R_{n-1}}\}.$$
Then $R_n\geq e^{\lambda R_{n-1}}$ and $R_n<e^{(\lambda+\delta)R_{n-1}}$ for all $n\in\mathbb{N}$, where we choose $\delta>0$ such that $\lambda+\delta<\tilde{\beta}$.
For $R>0$, set $$\Q_R:=\{Q(2^kR+2^{k-1}R+ij2^kR,2^kR,0):j\in\mathbb{Z},k\in\mathbb{N}_0\}.$$
This implies that $\Q_{R_{n+1}}\subset\Q_{R_n}$ for all $n\in\mathbb{N}_0$.

We define our sequence $(\A_n)$ as follows:
Let $Q_0\in\Q_{R_0}$. Set $$\A_0:=\{Q_0\}$$
and define inductively $$\A_n:=\left\{B\subset B'\in\A_{n-1}: B\subset\bigcup_{k\in\mathbb{Z}}T^k,F^n(B)\text{ is defined and }F^n(B)\in\Q_{R_n}\right\}.$$
First, we estimate the diameters of the sets in $\A_n$. Let $B\in\A_n$ and define $Q:=F^n(B)\in\Q_{R_n}$.
There is a unique sequence $(m_j,k_j)_{j=1}^n$ with $$F^n|_B=(F_{T_{m_n}^{k_n}})\circ\ldots\circ(F_{T_{m_1}^{k_1}})|_B.$$
Let $z_0$ be the center of $Q$. Since $Q$ is convex, we have by \eqref{a1} that
\begin{align*}
\text{diam}((G_{T_{m_n}^{k_n}})(Q))&\leq\sup\left\{\abs{(G_{T_{m_n}^{k_n}})'(z)}: z\in Q\right\}\cdot\text{diam}(Q)\leq\frac{4\pi}{R_n}\sqrt{2}R_n<18.
\end{align*}
(For the definition of $G_T$, see \eqref{G}.)
It is easy to see that $(F_{T_{m_{n-1}}^{k_{n-1}}}\circ\ldots\circ F_{T_{m_1}^{k_1}})^{-1}$ is defined on $D(G_{T_{m_n}^{k_n}}(z_0),18)$.
Hence, setting $g(z):=e^{\lambda R}$ and using that $$G_{T_{m_n}^{k_n}}(z_0)\in G_{T_{m_n}^{k_n}}(Q)\subset\{\Re z>R_{n-1}\},$$ we obtain
\begin{align*}
\text{diam}(B)&=\text{diam}((F_{T_{m_{n-1}}^{k_{n-1}}}\circ\ldots\circ F_{T_{m_1}^{k_1}})^{-1}(G_{T_{m_n}^{k_n}}(Q)))\\
&\leq36\cdot\sup\left\{\abs{((F_{T_{m_{n-1}}^{k_{n-1}}}\circ\ldots\circ F_{T_{m_1}^{k_1}})^{-1})'(z)}: z\in D(G_{T_{m_n}^{k_n}}(z_0),18)\right\}\\
&\leq\frac{36\cdot 4\pi}{R_{n-1}-18}\leq\frac{400}{g^{n-1}(R_0)}.
\end{align*}
Because $g^{n-1}(R_0)=\exp(E_\lambda^{n-2}(\lambda R_0))=\exp(E_\lambda^{n-2}(3\beta_\lambda))\geq E_\lambda^{n-1}(3\beta_\lambda)$, we deduce
$$\text{diam}(B)\leq\frac{1}{E_\lambda^{n-1}(2\beta_\lambda)}.$$

Now we estimate the density of $A_{n+1}$ in some set $B\in\A_n$ using
Theorem \ref{aspberg}. Let $B\in\A_n$, so that $Q:=F^n(B)\in\Q_{R_n}$. Again, there is a unique sequence $(m_i,k_i)$ with $Q=(F_{T_{m_n}^{k_n}})\circ\ldots\circ(F_{T_{m_1}^{k_1}})(B)$. We denote the inverse function of $(F_{T_{m_n}^{k_n}})\circ\ldots\circ(F_{T_{m_1}^{k_1}})|_B$ by $\phi_{Q,B}$.
Define $$S_Q:=\hspace{-0.5cm}\bigcup_{\overset{k\in\mathbb{Z},\hat{Q}\in\mathcal{Q}_{R_{n+1}}}{G_{T^k}(\hat{Q})\subset Q}}\hspace{-0.5cm}G_{T^k}(\hat{Q}).$$ Note that $$S_Q\subset\hspace{-0.2cm}\bigcup_{B'\in\mathcal{A}_{n+1}}\hspace{-0.2cm}F^n(B').$$ In fact, if $\hat{Q}\in\mathcal{Q}_{R_{n+1}}$ and $k\in\mathbb{Z}$ with $G_{T^k}(\hat{Q})\subset Q$ are given, then $$B':=\phi_{Q,B}(G_{T^k}(\hat{Q}))\subset B$$ and $F^{n+1}(B')=\hat{Q}$. So $B'\in\mathcal{A}_{n+1}$ and $F^n(B')=G_{T^k}(\hat{Q})$.
Lemma \ref{eigensch} implies
\begin{align*}
\text{dens}(A_{n+1},B)&=\text{dens}\left(\bigcup_{B'\in\A_{n+1}}B',B\right)\\
&\geq\frac{1}{L(\phi_{Q,B})^2}\cdot\text{dens}\left(\bigcup_{B'\in\A_{n+1}}F^n(B'),Q\right)\geq\frac{1}{L(\phi_{Q,B})^2}\cdot\text{dens}(S_Q,Q).
\end{align*}
Let $z_0$ be the center of $Q$ and choose $k\in\mathbb{N}_0$ such that $Q=Q(z_0,2^kR_n,0)$. Define $Q^*:=Q(z_0,2^kR_n-18,0)$. We introduce the notation $$U_R:=\left\{z\in\bigcup_{k\in\mathbb{Z}}T^k:\Re F(z)>R\right\}.$$
If $\hat{Q}\in\Q_{R_{n+1}}$ and $x\in S_Q\cap Q^*$, then it is immediate by the definition of $U_R$ that $x\in U_{R_{n+1}}\cap Q^*$. On the other hand, if $x\in U_{R_{n+1}}\cap Q^*$ and $k\in\mathbb{Z}$ with $x\in T^k$ are given, then $F(x)\in\hat{Q}$ for some $\hat{Q}\in\Q_{R_{n+1}}$. It follows that $G_{T^k}(\hat{Q})\cap Q^*\neq\emptyset$, and diam$(G_{T^k}(\hat{Q}))<18$ yields $G_{T^k}(\hat{Q})\subset Q$. So $x\in S_Q\cap Q^*$. The above considerations imply
$$S_Q\cap Q^*=T_{R_{n+1}}\cap Q^*.$$
Using this relation, we obtain
\begin{align*}
\text{dens}(A_{n+1},B)&\geq\frac{1}{L(\phi_{Q,B})^2}\cdot\text{dens}(S_Q,Q)=\frac{1}{L(\phi_{Q,B})^2}\cdot\frac{\abs{S_Q\cap Q}}{\abs{Q}}\\
&\geq\frac{1}{L(\phi_{Q,B})^2}\cdot\frac{\abs{S_Q\cap Q^*}}{\abs{Q^*}}\frac{\abs{Q^*}}{\abs{Q}}\geq\frac{1}{L(\phi_{Q,B})^2}\cdot\frac{\abs{U_{R_{n+1}}\cap Q^*}}{\abs{Q^*}}(1-\eps_2)\\
&\geq\frac{1}{L(\phi_{Q,B})^2}\cdot\text{dens}(U_{R_{n+1}},Q^*)(1-\eps_2).
\end{align*}
Let us estimate dens$(U_{R_{n+1}},Q^*)$. Let $x_1$ and $x_2$ be the minimal resp. maximal real part of points in $Q^*$. Applying \eqref{a} and the Cauchy-Schwarz inequality,
\begin{align*}
\abs{Q^*}&=(x_2-x_1)^2=\left(\int_{x_1}^{x_2}1 ds\right)^2=\left(\int_{x_1}^{x_2}\frac{\sqrt{\tilde{\theta}(s)}}{\sqrt{\tilde{\theta}(s)}}ds\right)^2\\
&\leq\int_{x_1}^{x_2}\tilde{\theta}(s)ds\cdot\int_{x_1}^{x_2}\frac{1}{\tilde{\theta}(s)}ds\leq\int_{x_1}^{x_2}\tilde{\theta}(s)ds\cdot\frac{1}{\pi}\left((\rho(f)+\eps_1)x_2+C'\right).
\end{align*}
Since $\lambda+\delta<\tilde{\beta}$, we have $\mathcal{L}(U_{R_{n+1}}\cap\{\Re z=s\})>\tilde{\theta}(s)$ for all $s\in[x_1,x_2]$. Further, at least $(1-\eps_2)(x_2-x_1)/2\pi=(1-\eps_2)x_2/4\pi$ of the $T^k$ intersect $Q^*$, so
\begin{align*}
\text{dens}(U_{R_{n+1}},Q^*)&\geq\left(1-\eps_2\right)\frac{x_2}{4\pi}\int_{x_1}^{x_2}\tilde{\theta}(s)ds\\
&\geq\left(1-\eps_2\right)\frac{x_2}{4\pi}\frac{\pi\abs{Q^*}}{\abs{Q^*}((\rho(f)+\eps_1)x_2+C')}\geq\frac{1-\eps_2}{4(\rho(f)+\eps_1)+\frac{C'}{x_2}}.
\end{align*}
Since $x_2$ is much bigger than $C'$ by hypothesis on $R_0$,
we obtain
\begin{equation}\label{1}
\text{dens}(U_{R_{n+1}},Q^*)\geq\frac{1}{4\rho(f)+\eps_3}.
\end{equation}

Now we show that $L(\phi_{Q,B})$ has an upper bound that is independent of $Q$ and $B$. This fact is well known and can be proved as follows: By Lemma \ref{eigensch}, we have $$L(\phi_{Q,B})\leq L(G_{T_{m_n}^{k_n}}|_Q)\cdot L(G_{T_{m_{n-1}}^{k_{n-1}}}|_{G_{T_{m_n}^{k_n}}(Q)})\cdot\ldots\cdot L(G_{T_{m_1}^{k_1}}|_{G_{T_{m_2}^{k_2}}\circ\ldots\circ G_{T_{m_n}^{k_n}}(Q)}).$$
Note that $G_{T_{m_n}^{k_n}}$ is defined on $D(z_0,3\cdot2^k\cdot R_n)\supset\overline{D(z_0,\sqrt{2}\cdot2^k\cdot R_n)}\supset Q$. Lemma \ref{first} yields that $L(G_{T_{m_n}^{k_n}}|_Q)\leq c_0$, independent of $Q$ and $F$.
Similarly, abbreviating $G_j:=G_{T_{m_j}^{k_j}}\circ\ldots\circ G_{T_{m_n}^{k_n}}$ for $j=2,\ldots,n$, we have $G_j(Q)\subset\{\Re z>R_0\}$ for all $j$. So $G_{T_{m_{j-1}}^{k_{j-1}}}$ is defined on
$$D(G_j(z_0),R_0)\supset D\left(G_j(z_0),\frac{1}{E_\lambda^{n-j+1}(2\beta_\lambda)}\right)\supset D(G_j(z_0),\text{diam }G_j(Q)).$$
Using Lemma \ref{first} again, one obtains $$L(G_{T_{m_{j-1}}^{k_{j-1}}}|_{G_{T_{m_j}^{k_j}}\circ\ldots\circ G_{T_{m_n}^{k_n}}(Q)})\leq\left(
1+\frac{2}{R_0E_\lambda^{n-j+1}(2\beta_\lambda)-1}\right)^4=\left(1+\frac{c_1}{E_\lambda^{n-j+1}(2\beta_\lambda)}\right)^4$$
for all $j=2,\ldots,n$.
It is clear that $$\prod_{k=1}^\infty\left(1+\frac{c_1}{E_\lambda^k(2\beta_\lambda)}\right)^4=c_2<\infty,$$
so
\begin{equation}\label{2}
L(\phi_{Q,B})\leq c_0c_2
\end{equation}
independent of $Q$ and $B$.
The formulas \eqref{1} and \eqref{2} imply that
\begin{equation*}
\text{dens}(A_{n+1},B)\geq\frac{c_3}{\rho(f)},
\end{equation*}
where $c_3$ does not depend on $n,Q$ or $f$. It follows by the functional equation \eqref{99} that
\begin{align*}
\Phi_\lambda(1/d_k)^\gamma\prod_{j=1}^k\Delta_j&\geq\Phi_\lambda(E_\lambda^{k-1}(2\beta_\lambda))^\gamma c_3^k\frac{1}{\rho(f)^k}\\
&=\Phi_\lambda(2\beta_\lambda)^\gamma\cdot\beta_\lambda^{(k-1)\gamma}c_3^k\frac{1}{\rho(f)^k}\\
&=\frac{\Phi_\lambda(2\beta_\lambda)^\gamma}{\beta_\lambda^\gamma}\left(\frac{\beta_\lambda^\gamma}{\rho(f)}c_3\right)^k.
\end{align*}
This term tends to $\infty$ as $k\to\infty$
whenever $\beta_\lambda^\gamma c_3/\rho(f)>1$, that is, $$\gamma>\frac{\log\rho(f)-\log c_3}{\log\beta_\lambda}.$$
For such values of $\gamma$, an application of Lemma \ref{infinity} yields $\mathcal{H}^{h_{\lambda,\gamma}}(I(F))=\infty$. As mentioned at the beginning of the proof, Lemma \ref{preserved} implies $\mathcal{H}^{h_{\lambda,\gamma}}(I(f))=\infty$ for the special parameter $\lambda$ that we chose. Finally, Theorem \ref{unabh} shows that this is true for every $\lambda\in(0,1/e)$.

\section{The estimate from above}

\subsection{Mittag-Leffler functions and the Besicovitch theorem}

Before we start with the proof of Theorem \ref{main2}, let us remark a simple fact. For $\rho>1/2$, we define
\begin{equation}\label{a2}
S_\rho:=\left\{\Im z\in\left[\frac{\pi}{2\rho},2\pi-\frac{\pi}{2\rho}\right](\text{mod }2\pi)\right\}.
\end{equation}
Note that we can find a constant $K=K(\rho)>0$ with the following property: If $B(z)=Q(z,K,\theta)$ is any square, then we can find disjoint squares $$B^j(z):=Q(w_j,2\pi(1-1/(2\rho)),0)\subset S_\rho\cap B(z),$$ $j=1,\ldots,k_1(B(z)),$ with
\begin{equation}\label{h}
\text{dens}\left(\bigcup_{j=1}^{k_1(B(z))} B^j(z),B(z)\right)\geq\left(1-\frac{1}{2\rho}\right)-\eps_1.
\end{equation}
For a given parameter $\rho>1/2$,
we define the \emph{Mittag-Leffler function with parameter $\rho$} by $$f_\rho(z):=\sum_{n=0}^\infty\frac{z^n}{\Gamma(\frac{n}{\rho}+1)}.$$
It is well-known that $\rho(f_\rho)=\rho$, which follows from the following representation of $f_\rho$ (see for example \cite[p. 83]{goldbergostrovskii}):
$$f_\rho(z)=
\begin{cases}
\rho\exp(z^\rho)+g_1(z)&,z\in U_\delta:=\left\{\abs{\text{arg}(z)}\leq\frac{\pi}{2\rho}+\delta\right\}\\
g_2(z)&,\abs{\text{arg}(z)}>\frac{\pi}{2\rho}
\end{cases},$$
where $0<\delta\leq\max\{\frac{\pi}{2\rho},(1-\frac{1}{2\rho})\pi\}$ and $g_i(z)=O(1/\abs{z})$ as $z\to\infty$, for $i=1,2$. It is not difficult to show that $f_\rho\in\mathcal{B}$, an argument can be found in \cite[section 4]{aspenbergbergweiler}.
Choose $C_0>0$ with $\abs{g_1(z)}\leq\frac{C_0}{\abs{z}}$ for all $z$. Let $R\gg\frac{C_0}{\sin(\delta)}$ and $R\gg K$ be so large that $\abs{g_2(z)}<1$ for all $z\in(\Com\setminus D(0,R))$ with $\abs{\text{arg}(z)}>\frac{\pi}{2\rho}$. Now choose $a>0$ so small that the function $f_{a,\rho}(z):=a f_\rho(z)$ satisfies $S(f_{a,\rho})\subset\mathbb{D}$ and $f_{a,\rho}(D(0,R))\subset\mathbb{D}$. The choice of $a$ implies that $f_{a,\rho}(z)\in\mathbb{D}$ for every $z\in\Com$ with $\abs{\text{arg}(z)}>\frac{\pi}{2\rho}$.
It follows that there is no logarithmic tract of $f_{a,\rho}$ meeting the set $S_\rho$.
In particular, $S_\rho\cap\mathcal{J}(F)=\emptyset$ if $F$ is the logarithmic transform of $f_{a,\rho}$.

By differentiating $F(z)=\log(f_{a,\rho}(\exp(z)))$, one obtains
$$F'(z)=\frac{f_\rho'(\exp(z))}{f_\rho(\exp(z))}\cdot\exp(z)$$
for all $z\in\mathcal{T}(f_{a,\rho})$.
Since $\abs{\text{arg}(\exp(z))}\leq\frac{\pi}{2\rho}$ for $z\in\mathcal{T}(f_{a,\rho})$, setting $w:=\exp(z)$ yields $\abs{w}\geq R$ and $D(w,\abs{w}\sin(\delta))\subset U_\delta$. By Cauchy's integral formula,
\begin{align*}
\abs{f_\rho'(w)-\rho^2w^{\rho-1}\exp(w^\rho)}&=\abs{g_1'(w)}=\abs{\frac{1}{2\pi i}\int_{\partial D(w,\abs{w}\sin(\delta))}\frac{g_1(\zeta)}{(w-\zeta)^2}d\zeta}\\
&\leq\frac{1}{2\pi}\frac{1}{\abs{w}^2\sin(\delta)^2}\frac{C_0}{\abs{w}}2\pi\abs{w}\sin(\delta)=\frac{C_0}{\abs{w}^2\sin(\delta)}
\end{align*}
and hence
\begin{equation}\label{c}
\abs{\frac{f_\rho'(w)w}{\rho^2w^\rho\exp(w^\rho)}-1}\leq\frac{C_0}{\rho^2w^\rho\exp(w^\rho)\sin(\delta)}\leq\eps_2.
\end{equation}
Further it is clear that
\begin{equation}\label{c2}
\abs{\frac{\rho\exp(w^\rho)}{f_\rho(w)}-1}\leq\frac{C_0}{\abs{w}}\leq\eps_2.
\end{equation}
The formulas \eqref{c} and \eqref{c2} imply
\begin{align}\label{d}
\abs{\frac{F'(z)}{\rho\exp(\Re z)^\rho}}&=\abs{\frac{F'(z)}{\rho w^\rho}}=\abs{\frac{f_\rho'(w)w}{f_\rho(w)\rho w^\rho}}\notag\\
&=\abs{\frac{f_\rho'(w)w}{\rho^2w^\rho\exp(w^\rho)}\frac{\rho\exp(w^\rho)}{f_\rho(w)}}\in(1-\eps_3,1+\eps_3)
\end{align}
for every $z\in\mathcal{T}(f_{a,\rho})$. We will use this estimate in the proof of Theorem \ref{main2}.

Another result that we will need is a version of the Besicovitch covering theorem which
follows easily from \cite[Theorem 5.4]{bliedtnerloeb}:
\begin{theo}\label{besi}
There exists a universal constant $N_0>0$ with the following property:
Let $A\subset\Com$. For every $z\in A$, let $r_z>0$ and $A_z=Q(z,r_z,0)$. Then there exists a countable subset $B\subset A$ such that $$A\subset\bigcup_{z\in B}A_z$$ and every $z\in A$ is contained in at most $N_0$ elements of $\{A_y: y\in B\}$.
\end{theo}

\subsection{Proof of Theorem \ref{main2}}

For $\rho>1/2$, let $a$ and $f_{a,\rho}$ be as in section 4.1, so that \eqref{d} holds for the logarithmic transform $F$ of $f_{a,\rho}$. In analogy to the proof of Theorem \ref{main1}, we show that $\mathcal{H}^{h_{\lambda,\gamma}}(\mathcal{J}(F))=0$ whenever $\rho$ is large enough. Lemma \ref{preserved} then implies that $\mathcal{H}^{h_{\lambda,\gamma}}(\mathcal{J}(f_{a,\rho}))=0$, since $\mathcal{J}(f_{a,\rho})\cap\mathbb{D}=\emptyset$ (compare the remark at the end of section 2.2.1).
Let $S_\rho$ be defined as in \eqref{a2} and recall that $S_\rho\cap\mathcal{J}(F)=\emptyset$.

Choose a point $z_0\in\mathcal{J}(F)$ and set $z_n:=F^n(z_0)$ for $n\in\mathbb{N}$.
Let $$B_n(z_0):=Q(z_n,K,\arg(F^n)'(z_0))$$ and let $\phi_{n,z_0}$ be the branch of $(F^n)^{-1}$ that maps $z_n$ to $z_0$. Because $\mathcal{T}(f)\subset\mathbb{H}$, $\phi_{n,z_0}$ is defined on $\mathbb{H}$ for all $n\in\mathbb{N}$,
so $K\ll R$ implies $$D(\phi_{n,z_0}|_{B_n(z_0)})\leq 1+\eps_4$$ uniformly in $n$ and $z_0$. Lemma \ref{second} yields
\begin{align}\label{xyxy}
Q_n^1(z_0)&:=Q(z_0,\abs{(\phi_{n,z_0})'(z_n)}K(1-\eps_5),0)\\
&\subset L_n(z_0):=\phi_{n,z_0}(R_n(z_0))\notag\\
&\subset Q(z_0,\abs{(\phi_{n,z_0})'(z_n)}K(1+\eps_5),0)=:Q_n^2(z_0),\notag
\end{align}
so
\begin{equation}\label{a3}
\frac{\abs{Q_n^1(z_0)}}{\abs{Q_n^2(z_0)}}\geq 1-\eps_6.
\end{equation}
Since
\begin{align*}
\text{dens}(F^n(Q_n^1(z_0)),B_n(z_0))&\geq \frac{1}{D(\phi_{n,z_0}|_{B_n(z_0)})^2}\text{dens}(Q_n^1(z_0),L_n(z_0))\\
&\geq(1-\eps_7)\text{dens}(Q_n^1(z_0),Q_n^2(z_0))\geq(1-\eps_8),
\end{align*}
we have
\begin{equation}\label{e}
\abs{F^n(Q_n^1(z_0))}\geq(1-\eps_8)\abs{B_n(z_0)}.
\end{equation}
Further,
\begin{equation}\label{z1}
F^n(Q_n^1(z_0))\supset Q(z_n,\abs{(F^n)'(z_0)}\abs{\phi_{n,z_0}'(z_n)}K(1-\eps_5)^2)=Q(z_n,(1-\eps_9)K).
\end{equation}
We assume that $B_n^1(z_0),\ldots,B_n^{k_2(B_n(z_0))}(z_0)$ are the $B_n^j$ that are contained in $F^n(Q_n^1(z_0))$. Then \eqref{z1} implies
\begin{equation}\label{a4}
\frac{k_2(B_n(z_0))}{k_1(B_n(z_0))}\geq1-\eps_{10}.
\end{equation}
Note that all of these estimates are uniform in both $z_0$ and $n$.

Let $$s_n(z_0):=\abs{\phi_{n,z_0}'(z_n)}K(1-\eps_5)=\frac{K(1-\eps_5)}{\abs{(F^n)'(z_0)}}$$ be the side length of $Q_n^1(z_0)$. By \eqref{d}, we have
\begin{align*}
s_{n+1}(z_0)&=\frac{(1-\eps_5)K}{\abs{(F^{n+1})'(z_0)}}=(1-\eps_5)\frac{1}{\abs{F'(F^n(z_0))}}\frac{K}{\abs{(F^n)'(z_0)}}\\
&=s_n(z_0)\frac{1}{\abs{F'(F^n(z_0))}}\geq(1-\eps_3)\frac{1}{\rho\exp(\Re F^n(z_0))^\rho}s_n(z_0)\\
&\geq(1-\eps_3)\frac{1}{\rho\exp(4\pi\abs{(F^n)'(z_0)})^\rho}s_n(z_0)\\
&=(1-\eps_3)\frac{1}{\rho\exp(\rho4\pi\abs{(F^n)'(z_0)})}s_n(z_0)\\
&\geq(1-\eps_3)\frac{1}{\rho\exp\left(\frac{\rho4\pi K(1+\eps_5)}{s_n(z_0)}\right)}s_n(z_0)\geq\frac{1}{\rho\exp\left(\frac{\rho(4\pi+\eps_{11})K}{s_n(z_0)}\right)}.
\end{align*}
Let $M\gg K$ be large and set $c:=4\pi+\eps_{11}$.
We fix a small number $r_0>0$ and define inductively $$r_{n+1}:=\frac{1}{\rho M\exp\left(\frac{\rho cK}{Mr_n}\right)}$$
for all $n$. For every $z\in\mathcal{J}(F)$, we claim that we can find $n(z)\in\mathbb{N}$ such that whenever $n\geq n(z)$, there exists $m_z(n)$ with $Mr_{n+1}\leq s_{m_z(n)}(z)\leq Mr_n$.
In fact, if $r_{n+1}\leq\frac{s_k(z)}{M}\leq r_n$, we find $$\frac{s_{k+1}(z)}{M}\geq\frac{1}{M\rho\exp\left(\frac{\rho cK}{s_k(z)}\right)}\geq\frac{1}{M\rho\exp\left(\frac{\rho cK}{Mr_{n+1}}\right)}=r_{n+2}.$$
We define $$\mathcal{Q}_n:=\{Q(z,r_n,0): z=(k+il)r_n\text{ for some }k,l\in\mathbb{Z}\}$$ to be the collection of squares in a $r_n$-mesh that covers the complex plane.

Let $Q:=Q(w,r_n,0)\in\mathcal{Q}_n$ be a square which meets $\mathcal{J}(F)$, where $$n\geq\max_{z\in\mathcal{J}(F)\cap Q}n(z).$$ We consider the square $$\tilde{Q}:=Q(w,(M+1)r_n,0).$$ Let us assign a square $Q_n(z)$ to every $z\in\tilde{Q}$ in the following way: If $z\in Q\setminus\mathcal{J}(F)$, let $Q_n(z):=Q(z,r,0)$ be a square with $Q_n(z)\cap\mathcal{J}(F)=\emptyset$. If $z\in\tilde{Q}\setminus Q$, let $Q_n(z):=Q(z,r,0)$ be such that $Q_n(z)\cap Q=\emptyset$. Finally, if $z\in Q\cap\mathcal{J}(F)$, let $$Q_n(z):=Q(z,s_{m_z(n)}(z),0)=Q_{m_z(n)}^1(z).$$
By Theorem \ref{besi}, there exist countable sets $X\subset Q\cap\mathcal{J}(F)$ and $Y\subset\tilde{Q}\setminus(Q\cap\mathcal{J}(F))$ such that $$\tilde{Q}\subset\bigcup_{z\in X}Q(z,s_{m_z(n)}(z))\cup\bigcup_{z\in Y}Q_n(z),$$ and every $z\in\tilde{Q}$ is contained in at most $N_0$ of these squares.
It follows that $$\abs{\mathcal{J}(F)\cap Q}=\abs{\mathcal{J}(F)\cap\bigcup_{z\in X}Q_n(z)}.$$ Let $z_0\in X$ and abbreviate $m:=m_{z_0}(n)$. Let $$L_m^j(z_0):=\phi_{m,z_0}(B_m^j(z_0))$$ for all $j=1,\ldots,k_1(B_m(z_0))$. Using small distortion again, we can find squares $Q_{L_m^j(z_0)}^k, k=1,2$, with $$Q_{L_m^j(z_0)}^1\subset L_m^j(z_0)\subset Q_{L_m^j(z_0)}^2$$ and $$\frac{\abs{Q_{L_m^j(z_0)}^1}}{\abs{Q_{L_m^j(z_0)}^2}}\geq 1-\eps_6$$ (compare \eqref{a3}). Let $w_j$ be the center of $B_m^j(z_0)$ and $v_j:=\phi_{m,z_0}(w_j)$. The side length $l(Q_{L_m^j(z_0)}^1)$ of $Q_{L_m^j(z_0)}^1$ satisfies
\begin{align*}
l(Q_{L_m^j(z_0)}^1)&\geq(1-\eps_5)\abs{(\phi_{m,z_0})'(w_j)}2\pi(1-1/(2\rho))=\frac{1-\eps_5}{\abs{(F^m)'(v_j)}}2\pi(1-1/(2\rho))\\
&\geq\frac{1-\eps_{12}}{\abs{(F^m)'(z_0)}}2\pi(1-1/(2\rho))
\end{align*}
by Lemma \ref{first}. Since $$\frac{1}{\abs{(F^m)'(z_0)}}=\frac{s_m(z_0)}{K(1-\eps_5)}$$ by the formula \eqref{xyxy}, we obtain
\begin{align*}
l(Q_{L_m^j(z_0)}^1)&\geq(1-\eps_{13})2\pi(1-1/(2\rho))\frac{s_m(z_0)}{K}\\
&\geq(1-\eps_{13})2\pi(1-1/(2\rho))\frac{M}{K}r_{n+1}\geq\tilde{M}r_{n+1},
\end{align*}
We have $\tilde{M}\gg 1$ since $M\gg K$. It follows that $Q_{L_m^j(z_0)}^1$ contains at least $\abs{Q_{L_m^j(z_0)}^1}/r_{n+1}^2$ elements of $\mathcal{Q}_{n+1}$. So the number $C_{m}^j(z_0)$ of squares in $\mathcal{Q}_{n+1}$ that are contained in $L_m^j(z_0)$ satisfies
\begin{equation}\label{f}
C_{m}^j(z_0)\geq(1-\eps_{14})\frac{\abs{Q_{L_m^j(z_0)}^1}}{r_{n+1}^2}\geq(1-\eps_{15})\frac{\abs{Q_{L_m^j(z_0)}^2}}{r_{n+1}^2}
\geq(1-\eps_{15})\frac{\abs{L_m^j(z_0)}}{r_{n+1}^2}
\end{equation}

Recall that the $B_m^j(z_0)$ do not intersect $\mathcal{J}(F)$ since $B_m^j(z_0)\subset S_\rho$. Hence we can estimate the number $N'(n,z_0)$ of the elements of $\mathcal{Q}_{n+1}$ that are contained in $Q_n(z_0)=Q_m^1(z_0)$ and that do not meet $\mathcal{J}(F)$:
\begin{align*}
B'(n,z_0)&\overset{\eqref{f}}{\geq}\sum_{j=1}^{k_2(B_m(z_0))}(1-\eps_{15})\frac{\abs{L_m^j(z_0)}}{r_{n+1}^2}\\
&=(1-\eps_{15})\frac{s_m(z_0)^2}{r_{n+1}^2}\sum_{j=1}^{k_2(B_m(z_0))}\frac{\abs{L_m^j(z_0)\cap Q_n(z_0)}}{\abs{Q_n(z_0)}}\\
&\geq(1-\eps_{16})\frac{s_m(z_0)^2}{r_{n+1}^2}\text{dens}\left(\bigcup_{j=1}^{k_2(B_m(z_0))}B_m^j(z_0),F^m(Q_n(z_0))\right)\\
&\overset{\eqref{a4}}{\geq}(1-\eps_{17})\frac{s_m(z_0)^2}{r_{n+1}^2}\frac{\abs{\bigcup_{j=1}^{k_1(B_m(z_0))}B_m^j(z_0)}}{\abs{F^m(Q_n(z_0))}}\\
&\overset{\eqref{e}}{\geq}(1-\eps_{18})\frac{s_m(z_0)^2}{r_{n+1}^2}\frac{\abs{\bigcup_{j=1}^{k_1(B_m(z_0))}B_m^j(z_0)}}{\abs{B_m(z_0)}}\\
&=(1-\eps_{18})\frac{s_m(z_0)^2}{r_{n+1}^2}\text{dens}\left(\bigcup_{j=1}^{k_1(B_m(z_0))}B_m^j(z_0),B_m(z_0)\right)\\
&\overset{\eqref{h}}{\geq}(1-\eps_{19})\left(1-\frac{1}{2\rho}\right)\frac{s_m(z_0)^2}{r_{n+1}^2}.
\end{align*}
It follows that the number $N(n,z_0)$ of the squares in $\Q_{n+1}$ that are sufficient to cover $\mathcal{J}(F)\cap Q_n(z_0)$ satisfies $$N(n,z_0)\leq \frac{1}{2\rho}(1+\eps_{19})\frac{s_m(z_0)^2}{r_{n+1}^2}.$$
Since all the squares $Q_n(z)$ with $z\in Y$ do not intersect $\mathcal{J}(F)\cap Q$, it follows that the number $N(Q)$ of squares in $\Q_{n+1}$ that can intersect $\mathcal{J}(F)\cap Q$ satisfies $$N(Q)\leq\frac{1}{2\rho}(1+\eps_{19})\frac{1}{r_{n+1}^2}\sum_{z\in X}(s_{m_z(n)}(z))^2.$$ Because every $y\in Q$ is contained in at most $N_0$ squares $Q_n(z)$ and $s_{m_z(n)}^2\leq M^2r_n^2$ for all $z$, this implies  $$N(Q)\leq\frac{M^2N_0}{2\rho}(1+\eps_{19})\frac{r_n^2}{r_{n+1}^2}.$$ In particular, the upper bound for $N(Q)$ only depends on $n$, but not on the particular location of $Q$ in the plane.
Repeating this argument, we get that $((1+\eps_{19})(M^2N_0)/(2\rho))^k(r_n/r_{n+k})^2$ elements of $\Q_{n+k}$ suffice to cover $\mathcal{J}(F)\cap Q$, for every $k\in\mathbb{N}$. This yields for fixed $\lambda\in(0,1/e)$ that
\begin{align*}
\mathcal{H}^{h_{\lambda,\gamma}}(Q\cap\mathcal{J}(F))&\leq\lim_{k\to\infty}\left((1+\eps_{19})\frac{M^2N_0}{2\rho}\right)^k\frac{r_n^2}{r_{n+k}^2}h_{\lambda,\gamma}(\sqrt{2}r_{n+k})\\
&=r_n^2\lim_{k\to\infty}\left((1+\eps_{19})\frac{M^2N_0}{2\rho}\right)^k\frac{1}{r_{n+k}^2}2r_{n+k}^2\Phi_\lambda(1/(\sqrt{2}r_{n+k}))^\gamma\\
&\leq 2r_n^2\lim_{k\to\infty}\left((1+\eps_{19})\frac{M^2N_0}{2\rho}\right)^k\Phi_\lambda(g_1^{n+k}(1/r_0)/\sqrt{2})^\gamma
\end{align*}
by \eqref{99}, where $g_1(x)=\rho M\exp(\rho cKx/M)$. Choose $x_1>x$ with $g_1^k(x)<E_\lambda^k(x_1)$ for all $k$. We obtain
\begin{align*}
\mathcal{H}^{h_{\lambda,\gamma}}(Q\cap\mathcal{J}(F))&\leq2r_n^2\lim_{k\to\infty}\left((1+\eps_{19})\frac{M^2N_0}{2\rho}\right)^k\Phi_\lambda(E_\lambda^{n+k}(x_1))^\gamma\\
&\leq2r_n^2\Phi_\lambda(x_1)^\gamma\beta_\lambda^{n\gamma}\lim_{k\to\infty}\left((1+\eps_{19})\frac{M^2N_0}{2\rho}\right)^k(\beta_\lambda^\gamma)^k.
\end{align*}
The above limit is zero if $$\beta_\lambda^\gamma<\frac{2\rho}{(1+\eps_{19})M^2N_0},$$ that is, $$\gamma<\frac{\log(2\rho)-\log(M^2N_0)-\log(1+\eps_{19})}{\log\beta_\lambda},$$ and this is completely independent of $Q$.
Since we can clearly cover $\mathcal{J}(F)$ by countably many squares $Q=Q(w,r_n,0)$ that satisfy $n\geq\max_{z\in\mathcal{J}(F)\cap Q}n(z)$, Theorem \ref{main2} now follows with the choices $K_2:=2\log M+\log N_0-\log 2+\log(1+\eps_{19})$ and $\rho_0>\exp(K_2)$.


\end{document}